\documentclass[english,12pt, letter]{article}
\usepackage{natbib}
 \bibpunct[, ]{(}{)}{,}{a}{}{,}%

\usepackage{amssymb,amsmath,amsthm,graphicx} 

\usepackage{comment,amsfonts,mathtools,bm,booktabs,lmodern, microtype,xcolor,algorithm}
\usepackage{enumerate}
\graphicspath{ {./} {./figures/} }
\usepackage[noend]{algorithmic}
\usepackage{hyperref}
\hypersetup{colorlinks=true,citecolor=red}
\usepackage{geometry}
\theoremstyle{definition}
\newtheorem{definition}{Definition}[section]
\newtheorem{remark}{Remark}[section]

\newtheorem{theorem}{Theorem}[section]
\newtheorem{lemma}{Lemma}[section]

\newtheorem{proposition}{Proposition}[section]
\DeclareMathOperator*{\argmax}{\arg\max}
\DeclareMathOperator*{\argmin}{\arg\min}

\def\noofprod{n}
\def\noofset{N}
\newcommand{\alsh}{\textsc{Assort-MNL(BZ)}}
\newcommand{\aheu}{\textsc{Assort-MNL(Approx)}} 
\newcommand{\ann}{\textsc{Assort-MNL}}
\DeclareMathOperator*{\apargmax}{\mathit{approx \ arg\,max}}

\newcommand{\cS}{\mathcal{S}}
\newcommand{\bv}{\textbf{v}}
\newcommand{\bz}{\textbf{z}}
\newcommand{\bp}{\textbf{p}}
\newcommand{\bu}{\textbf{u}}
\newcommand{\ba}{\textbf{a}}

\begin{document}
\title{Optimizing Revenue over Data-driven Assortments}

\author{Deeksha Sinha \\ Operations Research Center \\ Massachusetts Institute of Technology \\ email: \url{deeksha@mit.edu} \and Theja Tulabandhula \\ Information and Decision Sciences \\ University of Illinois at Chicago  \\ email: \url{tt@theja.org} }

\date{April 30, 2018}

\maketitle

\begin{abstract}
We revisit the problem of large-scale assortment optimization under the multinomial logit choice model without any assumptions on the structure of the feasible assortments. 
Scalable real-time assortment optimization has become essential in e-commerce operations due to the need for personalization and the availability of a large variety of items. While this can be done when there are simplistic assortment choices to be made, not imposing any constraints on the collection of feasible assortments gives more flexibility to incorporate insights of store-managers and historically well-performing assortments.
We design fast and flexible algorithms based on variations of binary search that find the revenue of the (approximately) optimal assortment.
We speed up the comparisons steps using novel vector space embeddings, based on advances in the information retrieval literature. 
For an arbitrary collection of assortments, our algorithms can find a solution in time that is sub-linear in the number of assortments and for the simpler case of cardinality constraints - linear in the number of items (existing methods are quadratic or worse). Empirical validations using the Billion Prices dataset and several retail transaction datasets show that our algorithms are competitive even when the number of items is $\sim 10^5$ ($100$x larger instances than previously studied).
\end{abstract}

\section{Introduction}
\label{sec:introduction}

Assortment optimization~\citep{kok2008assortment} is the problem of showing an appropriate subset (assortment) of items to a buyer taking into account their purchase (choice) behavior, and is a key problem studied in the revenue management literature. There are essentially two aspects to this problem: (a) the purchase behavior of the buyer, and (b) the metric that the seller wishes to optimize. Intuitively, the subset shown to the buyer impacts their purchase behavior, which in turn impacts the seller's desired metric such as conversion or revenue. Both the offline~\citep{Vel} and online~\citep{shipra} optimization settings have a wide variety of applications in retail, airline, hotel and transportation industries among others. Many variants of the problem have been extensively studied (hence, we only cite a few representative works) and is an integral part of multiple commercial offerings (for instance, see IBM Assortment Optimization~\citep{IBM} and JDA Category Management~\citep{JDA}).

The algorithms used for assortment optimization have several desired characteristics. These algorithms need to be \emph{computationally efficient} as well as \emph{data-driven} enabling real-time personalized optimization at scale. A significant motivation for focusing on the \emph{computational efficiency} of assortment optimization is the following: choice behavior varies across buyers as well as across time for the same buyer, and thus showing a single optimal assortment is not ideal. In fact, if the seller assumes a single static purchase behavior model for her target population and does not personalize, she may drastically lose out on many potential sales (especially of items not in the single assortment being displayed). As new information about the buyer (for instance, through click logs or some other form of feedback) updates their choice model, a real-time optimization of the most suitable assortment given the current model of the choice behavior is highly desirable. In particular, such an optimization scheme has to be scalable and extremely efficient to make personalization practical (e.g, see ~\cite{schurman2009user} for a study on the impact of delay on users in web search). 

To further motivate scalability, consider how global e-commerce firms like Flipkart or Amazon or Taobao/TMall display items. Every aspect of the page displayed to a customer that visits their pages is broken down into modular pieces with different teams (often as large as hundreds of employees) responsible for delivering and managing the functionality (see solutions such as ~\cite{diaz2012real,abbar2013real,shmueli2009best} and~\cite{wang2010ranking} addressing such real-time issues). To be specific, the assortment team may have a budget of at most hundreds of milliseconds to display the best possible assortment given the current profile of the customer. The delivery team may also have a similar budget to compute the feasibility of same-day delivery of the current items in the cart to the customer's location in specific time windows. Given such time budgets, it is imperative that teams work on highly scalable algorithms to deliver the best experience and maximize their goals (e.g., revenue in the case of assortment optimization, cost in the case of delivery optimization). In practice the current state of the art for assortment planning tends to be based on memorization/look-up-table like setups where a list of globally popular items are, more often than not, greedily bunched into an assortment to be displayed to the customer. To further illustrate, we also mention another prototype application where such real-time computations are highly desired. In map products such as Google Maps or Apple Maps, there is a feature showing the best local restaurants and coffee-shops in the neighborhood of the customer's location. Given that there would typically be many such businesses in a neighborhood and a limited screen space, tailoring the assortment of businesses to show according the customer's propensity to click is desired, and doing this in real-time can also be extremely challenging. 

Another desired characteristic of assortment optimization algorithms is their ability to perform well with different types of constraints on the set of feasible assortments. One of the most commonly encountered and well studied constraint is that the assortment size is bounded. Nonetheless, in many situations, the set of feasible assortments can have a more general description than this. For instance, frequent itemsets~\citep{borgelt2012frequent} discovered using transaction logs can readily give us candidate \emph{data-driven} assortments, that can lead to computationally difficult assortment optimization problem instances. There is a rich history of frequent itemset mining, both in research and in practice for retail data. Although, no previous work has connected them to assortment planning, we believe they are a natural fit and aligned with the seller's objective of maximizing expected revenue. This is because frequent itemsets zero in on bundles of items that have been most often purchased together by a customer. This means that these assortments of items are naturally `best selling'. But often these sets are very large in number and all of them cannot be displayed to the customers due to various infrastructural or other business related constraints. As a result, we would prefer choosing one (or more) of them that have the potential to yield high expected revenue (assuming a probabilistic model of purchase behavior given these assortments such as the MNL model). This selection of one assortment among these data-driven feasible assortments turns out to be quite difficult because their collection cannot be easily represented using a compact (integer) polytope. This latter property is key for efficient algorithms, as seen in previous works. For instance, if assortment planning is carried over the polytope of all itemsets whose size is bounded, then under the MNL choice model, efficient algorithms proposed in~\cite{capMNL,jagabathula2014assortment} can be used (generalizations to unimodular polytopes have also been studied). In reality, we would like to specify data-driven itemsets, such as above, as candidate assortments in our optimization problem. Store managers can also curate arbitrary assortments based on domain knowledge and other business constraints. Unfortunately, no existing algorithms (including methods for integer programming) work well when these sets are not compactly represented.  By compact, we mean a polynomial-sized description of the collection of feasible assortments; for instance, a polytope in $\noofprod$ dimension ($\noofprod$ is the number of items) with at most a polynomial number of facets. 

While no general methods other than those for integer programming exist when the feasibility constraints are not unimodular, the special case of capacity constrained setting (where we consider all assortments with size less than a threshold) has been well studied algorithmically. The key approaches in this special case are: (a) a linear programming (LP) based approach by ~\cite{davis2013assortment} (time complexity O($\noofprod^{3.5}$) if based on interior point methods), (b) the \textsc{Static-MNL} algorithm  (time complexity O($\noofprod^2\log\noofprod$)) by ~\cite{capMNL}, and (c) ADXOpt (time complexity $O(n^2bC)$, where $n$ is the number of items, $C$ is the maximum size of the assortment and $b = \min \{C, n-C+1 \}$ for the MNL choice model) by ~\cite{jagabathula2014assortment}. Note that ADXOpt is a local search search heuristic designed to work with many different choice models, but is not guaranteed to find the optimal when a general collection of feasible assortments are considered under the MNL model. Offline methods that work with other purchase behavior models include a mixed integer programming approach in ~\cite{Vel} (NP-hard) for the distribution over ranking model, ~\cite{davis2014assortment} for the nested logit, ~\cite{bront2009column,rusmevichientong2014assortment} for the mixture of MNLs model, and ~\cite{desir2015capacity} for the Markov chain model to name a few. Capacity and other business constraints have also been considered for some of these  approaches (see~\cite{davis2013assortment}). It is not clear how to extend these methods when the feasible assortments cannot be compactly represented. Finally, note that although one could look at continuous estimation and optimization of assortments in an online learning framework~\citep{shipra} that can potentially scale and be data-driven, decoupling these two operations brings in a lot of flexibility for both steps, especially because the optimization approaches can be much better tuned for performance. Even when one is interested in an online learning scheme, efficient algorithms, such as the ones we propose, can easily be used as subroutines to regret minimizing online assortment optimization schemes.

In this work, we focus on a popular parametric purchase behavior model called the Multinomial Logit (MNL) model, and assume that the seller's objective is to maximize expected revenue by choosing the right assortment (prices are assumed known and fixed). We design algorithms to perform assortment optimization that are \emph{computationally efficient} and \emph{data-driven/flexible} with respect to the candidate assortments. These algorithms, namely \ann,  \aheu \ and \alsh, build on (noisy) binary search and make use of efficient data structures for similarity search, both of which have not been used for assortment optimization before. The consequence of these choices is that we can solve problem instances with \emph{extremely general specification of candidate assortments} in a \emph{highly scalable manner}. By leveraging recent advances in similarity search, a subfield of data mining, our algorithms can solve fairly large instances ($\sim 10^5$ items) within reasonable computation times even when the collection of feasible assortments have no compact representations. This allows store managers to seamlessly add or delete assortments, while still being able to optimize for the best assortments to show to the customers. In the capacity-constrained assortment setting, not only do our methods become more time (\textit{linear} in the number of items) and memory efficient in theory, they are also better empirically as shown in our experiments. Even in the setting with general assortments, our algorithms compute (nearly) optimal assortments given collections of assortments as large as $\sim 10^4$ within 2 seconds on average. These experiments are carried out using prices from the Billion Prices dataset~\citep{IAH6Z6_2016} as well as using frequent itemsets mined from transaction logs~\citep{borgelt2012frequent}. 

In summary, our work addresses the gap of practical assortment planning at Internet scale (where we ideally seek solutions in 10s-100s of milliseconds) and is complementary to works such as~\cite{Vel}, which focus on richer choice models (e.g., the distribution over rankings model, Markov chain model etc.), that lead to challenging integer programming problems. These latter solutions are also limited to instances where the sets can be efficiently described by a polytope. Fixing the choice model to be the MNL model allows our algorithms to scale to practical instances, especially those arising in the Internet retail/e-commerce settings as described above. The rest of the paper is organized as follows. In Section~\ref{sec:preliminaries}, we describe some preliminary concepts. Our proposed algorithms are in Section~\ref{sec:optimize}, whose performance we empirically validate in Section~\ref{sec:experiments}. Finally, Section~\ref{sec:conclude} presents some concluding remarks and avenues for future work.

\section{Preliminaries}
\label{sec:preliminaries}
\vspace{2mm}
\subsection{Assortment Planning}

The assortment planning problem concerns with choosing the assortment among a set of feasible assortments that maximizes the expected revenue (note that we use price and revenue interchangeably throughout the paper). Without loss of generality, let the items be indexed from $1$ to $\noofprod$ in the decreasing order of their prices, i.e., $p_1 \geq p_2 \geq \cdots p_{\noofprod}$. Let $f(A) = \sum_{l \in A} p_l \mathbb{P}(l|A)$ denote the revenue of the assortment $A \subseteq \{1,...,\noofprod\}$. Here $\mathbb{P}(l|A)$ represents the probability that a user selects item $l$ when assortment $A$ is shown to them and is governed by a choice model (we briefly remark about optimizing over multi-item choice models in Section~\ref{sec:conclude}). The expected revenue assortment optimization problem is: $\max_{A} f(A)$. 

Typically, in addition to choice models, the compactness of representation of feasible assortments also influences the computational tractability of the problem. In Section~\ref{sec:optimize}, we develop efficient algorithms for assortment optimization that do not require any structural assumption on the feasible sets, nor do they require a compact representation. The algorithms only require a set $\mathcal{S}$ of all the feasible assortments as input. Furthermore, when some structure is present, for instance, when we want to optimize over all assortments of a given size range (i.e. constraints of the type $c \leq |A| \leq C$), we can improve the time and space complexity of our algorithms significantly, matching existing solutions in the literature.

\subsection{The Multinomial Logit Model}

There are many choice models that have been proposed in the literature~\citep{train2009discrete} including Multinomial Logit (MNL), the mixture of MNLs (MMNL) model, the nested logit model, multinomial probit, the distribution over rankings model~\citep{farias2013nonparametric}, exponomial ~\citep{alptekinouglu2016exponomial} choice model and the Markov chain model~\cite{desir2015capacity}. The MNL model is arguably one of the most commonly used choice model, due to its simplicity. Though some of the above choice models are more expressive than the MNL choice model, they are harder to estimate reliably from data as well as make the corresponding optimization problems NP-hard. Thus, in our work we consider the MNL model to be the choice model.

Let the MNL choice model~\citep{luce1960individual} parameters be represented by a vector $\mathbf{v} = \left(v_0, v_1, \cdots v_{\noofprod}\right)$ with $0 \leq v_i \leq 1 \;\;\forall i$. Parameter $v_i, \ 1\leq i \leq \noofprod$, captures the preference of the user for purchasing item $i$. To be precise, $\log(v_i)$ is the  mean utility derived from item $i$ (MNL is an instance of the random utility maximization framework where utilities are random variables that are Gumbel distributed). Similarly, parameter $v_0$ captures the preference for not making any purchase or selection. For this model, it can be shown that $\mathbb{P}(l|A) = \frac{v_l}{v_0 + \sum_{l' \in A} v_{l'}}$. Intuitively, the probability with which a buyer will pick an item increases when the item is shown with other items with lower mean utilities, and vice versa. Owing to its structure, the MNL model leads to tractable assortment optimization problems when the feasible sets can be represented compactly. For instance, in~\citet[Section 5, Table 11]{Vel}, the authors show that capacity-constrained assortment optimization under the MNL model can be solved fairly quickly (more than $5$ times faster than a proposed integer programming formulation), with relatively small gap from the true optimal revenue (a Mixed Multinomial Logit model was used as the ground truth). Coupled with the fact that estimating MNL parameters is relatively easy, this evens out some of its shortcomings such as: (a) under-fitting the data, and (b) satisfying the independence of irrelevant alternatives property.

\subsection{Locality Sensitive Hashing (LSH)}\label{subsec:LSH}

LSH~\citep{andoni2008near,har2012approximate} is a technique for finding vectors from a known set of vectors (referred to as the search space), that are  `similar' (i.e, \emph{neighbors} according to some metric) to  a given query vector in an efficient manner. It uses hash functions that produce similar values for input points (vectors) which are  similar to each other as compared to points which are not.  

We first formally define the closely related (but distinct) problems of finding the \textit{near neighbor} and the \textit{nearest neighbor} of a point from a set of points denoted by $P$ (with $|P| = N$).

\begin{definition} The $(c,r)$-NN (approximate \emph{near} neighbor) problem with failure probability $f \in (0,1)$ is to construct a data structure over a set of points $P$ that supports the following query:  given point $q$, if $\min_{p \in P}d(q,p) \leq r$, then report some point $p' \in P\cap \{p: d(p,q) \leq cr\}$ with probability $1-f$. Here, $d(q,p)$ represents the distance between points $q$ and $p$ according to a metric that captures the notion of neighbors. Similarly, the $c$-NN (approximate \emph{nearest} neighbor) problem with failure probability $f \in (0,1)$ is to construct a data structure over a set of points $P$ that supports the following query: given point $q$, report a $c$-approximate nearest neighbor of $q$ in $P$ (i.e., return $p'$ such that $d(p',q) \leq c\min_{p \in P}d(p,q)$) with probability $1-f$.
\end{definition}

As mentioned earlier, the data structures that can solve the above search problems efficiently are based on hash functions, which are typically endowed with the following properties:

\begin{definition}\label{def:hashfunctions}
A $(r,cr,P_1,P_2)$-sensitive family of hash functions ($h \in \mathcal{H}$) for a metric space $(X,d)$ satisfies the following properties for any two points $p,q \in X$:
\begin{itemize}
\item If $d(p,q) \leq r$, then $Pr_{\mathcal{H}}[h(q) = h(p)] \geq P_1$, and
\item If $d(p,q) \geq cr$, then $Pr_{\mathcal{H}}[h(q) = h(p)] \leq P_2$.
\end{itemize}
\end{definition}

In other words, if we choose a hash function uniformly from $\mathcal{H}$, then they will evaluate to the same value for a pair of points $p$ and $q$ with high ($P_1$) or low ($P_2$) probability based on how similar they are to each other. The following theorem states that we can construct a data structure that solves the approximate near neighbor problem in sub-linear time.

\begin{theorem} [\cite{har2012approximate} Theorem 3.4] Given a $(r,cr,P_1,P_2)$-sensitive family of hash functions, there exists a data structure for the $(c,r)$-NN (approximate near neighbor problem) over points in the set $P$ (with $|P| = N$) such that the time complexity of returning a result is $O(nN^\rho/P_1 \log_{1/P_2}N)$ and the space complexity is $O(nN^{1+\rho}/P_1)$. Here $\rho = \frac{\log 1/P_1}{\log 1/P_2}$. Further, the failure probability is upper bounded by $1/3 + 1/e$.\label{thm:near-neighbor}
\end{theorem}

\begin{remark} Note that the failure probability can be changed to meet to any desired upper bound by \emph{amplification}: here we query a constant number (say $\kappa$) of different copies of the data structure above and aggregate the results to output a single near-neighbor candidate. Such an output will fail to be a $(c,r)$-NN with probability upper bounded by $(1/3+1/e)^{\kappa}$. 
\end{remark}

While the hash family with the properties stated in Definition~\ref{def:hashfunctions} already gives us a potential solution for the $(c,r)$-NN problem, the following data structure allows for a finer control of the time complexity and the approximation guarantees (corresponding to Theorem~\ref{thm:near-neighbor}). In particular, as described in~\cite{andoni2008near}, we employ multiple hash functions to increase the confidence in reporting near neighbors by amplifying the gap between $P_1$ and $P_2$. The number of such hash functions is determined by suitably chosen parameters $L_1$ and $L_2$. In particular, we can choose $L_2$ functions of dimension $L_1$, denoted as $g_j(q) = ( h_{1,j}(q) , h_{2,j}(q), \cdots h_{L_1, j}(q) )$, where $h_{t,j}$ with $1 \leq t \leq L_1 ,1 \leq j \leq L_2$ are chosen independently and uniformly at random from the family of hash functions. The data structure for searching points with high similarity is constructed by taking each point $x$ in the search space and storing it in the location (bucket) indexed by  $g_j(x) , 1 \leq j \leq L_2$.  When a new query point $q$ is received,  $g_j(q) , 1 \leq j \leq L_2$ are  calculated and all the points from the search space in the buckets  $g_j(q) , 1 \leq j \leq L_2$ are retrieved.  We then compute the similarity of these points with  the query vector in a sequential manner and return any point that has a similarity greater than the specified threshold $r$. We also interrupt the search after finding the first $L_3$ points including duplicates for a suitably chosen value of $L_3$ (this is necessary for the guarantees in Theorem~\ref{thm:near-neighbor} to hold).

Given $P_1$, the probability that any point  $p$ such that  $d(q,p) \leq cr$ is retrieved is at least $ 1 - (1 - P_1^{L_1})^{L_2} $. Thus for any desired error probability $\delta > 0$, we can choose  $L_1$ and $L_2$ such that $p$ is returned by the data structure with probability at least $1 - \delta$. The time taken to perform a query, i.e., retrieve a point with the highest similarity can be as low as $O(nN^\rho)$, where $\rho = \frac{\ln (1/P_1)}{\ln (1/P_2)}$ is a number less than $1$, $N$ is the number of points in the search space and $n$ is the dimensionality of each point (see Theorem~\ref{thm:near-neighbor} above, also ~\cite{wang2014hashing} and~\cite{wang2016learning}). The key insight that allows for this sublinear dependence on the number of points is due to the appropriate choices of $L_1$, $L_2$ and $L_3$, which determine collisions of hashed (i.e., projected) transformations of the original points. In particular, choosing $L_1 = \log N$, $L_2 = N^\rho$ and $L_3 = 3L_2$ gives the desired computational performance (in addition to influencing the error probabilities). This sublinearity of retrieval time is what drives the computational efficiency of hashing based retrieval methods including our algorithms (in Section \ref{sec:optimize}).  For instance, Uber 
\citep{uberBlog} has used an LSH based similarity search technique for fraudulent trip detection for its app-based ride-hailing service, and their system is able to achieve detection results much faster using hashing (4 hours) compared to naive linear search (55 hours). Finally, note that the memory requirements of the data structure turns out to be $O(nN^{1+\rho})$ (as stated in Theorem~\ref{thm:near-neighbor}), which is not too expensive in many large scale settings (we need $O(nN)$ space just to store the points).

It turns out that out the problem we care about in this work is related to the \emph{nearest} neighbor problem. And the nearest neighbor problem has a close connection to the \emph{near} neighbor problem described above. In particular, the following theorem states that an approximate nearest neighbor data structure can be constructed using an approximate near neighbor subroutine without taking a large hit in the time complexity, the space complexity or the retrieval probability.

\begin{theorem} [\cite{har2012approximate} Theorem 2.9] Let P be a given set of N points in a metric space, and let $c, f \in (0,1)$ and $\gamma \in (\frac{1}{N},1)$ be parameters. Assume we have a data-structure for the $(c,r)$-NN (approximate \emph{near} neighbor) problem that uses space $S$ and has query time $Q$ and failure probability $f$. Then there exists a data structure for answering $c(1+O(\gamma))$-NN (approximate \emph{nearest} neighbor) problem queries in time $O(Q\log N)$ with failure probability $O(f\log N)$. The resulting data structure uses $O(S/\gamma \log^2 N)$ space.\label{thm:nearest-neighbor}
\end{theorem}

More intuition on how a collection of near neighbor data structures can be used to build a nearest neighbor data structure is given in the electronic companion. Finally, we remark about the properties of the hash families characterized by parameters $\rho$ and $c$. These two parameters are inversely related to each other. For instance, in \cite{andoni2008near}, the authors propose a family $\mathcal{H}$ for which $\rho(c) = \frac{1}{c^2} + O(\log\log N / \log^{1/3}N)$. For large enough $N$, and for say $c = 2$ (2 approximate near neighbor), $\rho(c) \approx .25$, and thus the query completion time of near neighbor (as well as nearest neighbor) queries is $\propto N^{.25}$ which is a significant computational saving.

\subsection{Maximum Inner Product Search (MIPS)}

The MIPS problem is that of finding the vector in a given set of vectors (points) which has the highest inner product with a query vector. Precisely, for a query vector $q$ and set of points $P$, the optimization problem is: $ \max_{x \in P} q  \cdot  x$ (the `$\cdot$' operation stands for inner product). When there is no structure on $P$, one can solve for the optimal via a linear scan, which can be quite slow for large problem instances. To get around this, we can also solve the instance approximately using methods based on Locality Sensitive Hashing (LSH) and variants~\citep{neyshabur2015symmetric}. See the electronic companion for one such reduction. Approximate methods have also been proposed for related problems such as the Jaccard Similarity (JS) search in the information retrieval literature. For instance, one can find a set maximizing Jaccard Similarity with a query set using a technique called \textsc{Minhash} \citep{broder1997resemblance}. The $c$-NN problem, as discussed above, can be addressed using \textsc{L2LSH}~\citep{datar2004locality} and many other solution approaches (see for instance~\cite{li17d} and references within). The MIPS problem can be solved approximately using methods such as \textsc{L2-ALSH(SL)}~\citep{shrivastava2014asymmetric} and \textsc{Simple-LSH}~\citep{neyshabur2015symmetric}.
MIPS will be a key part of all of our algorithms, namely \ann{}, \aheu{} and \alsh{}. Among them, \ann{} relies on a subroutine that solves the MIPS problem exactly.

Let the approximation guarantee for the nearest neighbor obtained be $1+\nu$ (i.e., set the parameters of the data structure in Theorem~\ref{thm:nearest-neighbor} such that it solves the $(1+ \nu)$-NN problem). Then the following straightforward relation holds.
\begin{lemma}\label{lemma:nn2mips}
If we have an $1+\nu$ solution $x$ to the nearest neighbor problem for vector $y$, then $1 + (1+\nu)^2(\max_{p \in P}p\cdot y - 1) \leq x \cdot y \leq \max_{p \in P}p\cdot y$.
\end{lemma}
The above lemma will be key in designing two of our algorithms, namely \aheu{} and \alsh, in the next section.
\section{New Algorithms for Assortment Planning}
\label{sec:optimize}

We propose three new algorithms: \ann, \aheu \  and \alsh . \ann \ aims to find an assortment with revenue that is within an additive tolerance $\epsilon$ of the optimal revenue. It turns out that a key computational step in \ann \ can be 
made faster using an LSH based subroutine. As a tradeoff, such a subroutine typically produces approximate results and has a failure probability associated with it. Our second algorithm \aheu \ addresses the approximation aspect, while the third algorithm \alsh \ addresses the probabilistic errors. Along the way, we also present an optimized version of \ann \ when the feasible set of assortments is given by capacity constraints.

\subsection{First Algorithm: \ann{}}

\begin{figure}

\noindent\begin{minipage}{\textwidth}
\hspace{\dimexpr-\fboxrule-\fboxsep\relax}\fbox{%
\begin{minipage}[t]{.48\textwidth}
\vspace{0pt}
\begin{algorithm}[H]
\caption{\ann{}}
\label{alg:ann_outline}
\begin{algorithmic}[1] 
\REQUIRE{ Prices $\{p_i\}_{i=1}^{n}$, model parameter \bv, tolerance parameter $\epsilon$, set of feasible assortments $\cS$ } \\
\STATE{$L_1 = 0, U_1 = p_1 , t=1, S^* = \{1 \} $}\\ 
\WHILE{ $U_j - L_j > \epsilon$} 
\STATE{ $K_j = (L_j + U_j)/2 $} \\
\IF{$K_j \leq  \max_{S \in \cS } f(S, \bv)$ \label{alg:compare-step}} 
\STATE{$L_{j+1} =  K, U_{j+1} = U_j$}
\STATE{Pick any $S^* \in \{ S:f(S, \bv) \geq K \}$}
\ELSE 
\STATE{$L_{j+1} = L_j , U_{j+1} = K$} 
\ENDIF 
\ENDWHILE
\RETURN{$ S^*$}
\vspace{29mm}
\end{algorithmic}
\end{algorithm}
\end{minipage}
}
\hspace{\dimexpr-\fboxrule-\fboxsep\relax}\fbox{%
\begin{minipage}[t]{.48\textwidth}
\vspace{0pt}\raggedright
\begin{algorithm}[H]
\caption{\textsc{Compare step} in \ann{} }
\label{alg:ann_comp}
\begin{algorithmic}
\STATE{Given comparison: \vspace{-6mm} $$K \leq \max_{S \in \cS } \frac{1}{v_0} \ \sum_{i \in S}  v_i (p_i - K)$$}
\vspace{-10mm}
\STATE{Formulate an equivalent MIPS instance with:
\vspace{-4mm}
\begin{align*}
\bp &:= \left( p_1, p_2 \cdots p_{\noofprod} \right) \\
\mathbf{\hat{v}_K} &:= (v_1, \cdots, v_n, -v_1K,-v_2K, \cdots -v_{\noofprod}K) \\
\bu^S &:= (u_1, u_2, \cdots u_n) \ \text{ where } \  u_i = \mathbf{1} \{ i \in S\}\\
\mathbf{U} &:=  \{ \mathbf{u}^S | \ S \in \cS \}\\
\hat{\bz}^S &:= \left( \bp \circ \bu^S, \bu^S \right) \\
\mathbf{\widehat{Z}}& := \{ \hat{\mathbf{z}}^S : S \in \cS \}
\end{align*}
}\\
\STATE{Solve the MIPS instance: \\
$\quad\quad\quad \bz^{\tilde{S}} \in \argmax_{\hat{\bz}^S \in \mathbf{\widehat{Z}}   } \mathbf{\hat{v}_K} \cdot \mathbf{\hat{z}}$}
\STATE{Output result of an equivalent comparison: \\
$\quad\quad\quad  K \leq  \frac{\bv \cdot \bz^{\tilde{S}}}{v_0} $}
\end{algorithmic}
\end{algorithm}
\end{minipage}
}
\end{minipage}

\end{figure}


\ann \ (Algorithm~\ref{alg:ann_outline}) aims to find an $\epsilon$-optimal assortment i.e. an assortment with revenue within a small interval (defined by tolerance parameter $\epsilon$) of the optimal assortment's revenue.  In this algorithm, we search for the revenue maximizing assortment using the following iterative procedure. In each iteration, we maintain a search interval and check if there exists an assortment with revenue greater than the mid-point of the search interval. Then, we perform a binary search update of the search interval i.e. if there exists such an assortment then the lower bound of the search interval is increased to the mid-point (and this assortment is defined as the current optimal assortment). Otherwise, the upper bound is decreased to the mid-point. We continue iterating and narrowing down the search space until its length becomes less than the tolerance parameter $\epsilon$. Although it seems that the binary search loop is redundant if the comparison step is solving an assortment optimization problem, we will show soon that the comparison can be reformulated such that the overall time complexity of \ann \ is small. 

The search interval starts with the lower and upper bounds ($L_1$ and $U_1$) as $0$ and $p_1$ respectively, where $p_1$ is the  highest price among all items (note that $p_1$ is an upper bound on the revenue of any assortment.) 
Before the start of the algorithm, the optimal assortment is initialized to the set $\{ 1 \}$. This initial assortment will be returned as the optimal only when all the assortments have revenue less than $\epsilon$ and in that case this assortment is approximately optimal.

The strength of \ann \ is in its ability to efficiently answer a transformed version of the comparison $K \leq  \max_{S \in \cS } f(S, v)$, which is checking if there exists an assortment with revenue greater than the mid-point ($K$) of the current search interval. We refer to this step as the \textsc{Compare-Step} (see line~\ref{alg:compare-step} of Algorithm~\ref{alg:ann_outline}).
In particular, the original decision problem is transformed as follows.  In every iteration, we need to check if there exists a set $S \in \cS$  such that 
\begin{align*}
    K \leq \frac{\sum_{i \in S} p_i v_i}{v_0 + \sum_{i \in S} v_i} 
\Leftrightarrow &    K \leq \frac{1}{v_0} \ \sum_{i \in S}  v_i (p_i - K).
\end{align*}
This is equivalent to evaluating if $
 K \leq \max_{S \in \cS } \frac{1}{v_0} \ \sum_{i \in S}  v_i (p_i - K)$, allowing us to focus on the transformed optimization problem: 
 \vspace{-5mm}
\begin{equation}
\argmax_{S \in \cS}  \ \sum_{i \in S}  v_i (p_i - K).
\label{opt_prob}
\end{equation}
Now, we will describe how this optimization problem can be solved efficiently when: (a) there is a compact representation, for example, the capacitated setting, and (b) there is no compact representation for the set of feasible assortments.

\subsubsection{Assortment Planning with Capacity Constraints} \label{subsec:capacity-constrained-modification}
We first start with a case where there is structure on the constraints defining the set of feasible assortments. In particular, consider the well studied capacity-constrained setting with $\cS = \{ S:  |S| \leq C  \}$, where constant $C$ specifies the maximum size of feasible assortments.  The key insight here is that the operation $ \argmax_{S \in \cS}  \ \sum_{i \in S}  v_i (p_i - K)$ can be decoupled into problems of smaller size, each of which can be solved efficiently. This is because the problem can be interpreted as that of finding a set of at most $C$ items that have the highest product $v_i(p_i - K)$ and the value of this product is positive. Thus, to solve this optimization problem we only need to calculate the value $v_i(p_i - K)$ for each item $i$ and sort these values. This strategy can also be extended to some other capacity-like constraints as described below:\\
\noindent(1) \textit{Lower bound on assortment size} - This constraint requires that all feasible assortments have at least $c \leq C$ items in addition to having at most $C$ items. In every iteration of Algorithm \ref{alg:ann_outline}, we find the top-$C$ items (by the  product $v_i(p_i - K)$)  that give a positive value of the inner product. This can be modified to finding the top-$C$ items and including at least the top-$c$ of them irrespective of the sign of the product. \\
\noindent(2) \textit{Capacity constraints on subsets of items} - To ensure diversity in the assortment, we may have constraints on how many items can be chosen from certain subsets of items. Suppose the items are partitioned into subsets $B_1, B_2, \cdots, B_w$ and we have capacity constraints $C_1, C_2, \cdots, C_w$ on each subset respectively, then the comparison to be solved in every iteration can be written as: 
	\begin{eqnarray*}
		K \leq   \max_{S: |S| \leq C_1, S \subseteq B_1} \frac{1}{v_0}   \sum_{i \in S}  v_i (p_i - K)  + \cdots +   \max_{S: |S| \leq C_w, S \subseteq B_w} \frac{1}{v_0}  \sum_{i \in S}  v_i (p_i - K). 
	\end{eqnarray*} 
One can solve these $w$ independent problems the same way as the previous setting.    
\subsubsection{Assortment Planning with General Constraints}
In the absence of structure in the set of feasible assortments, we cannot decouple the optimization problem in Equation (\ref{opt_prob}) as before. Nonetheless, this can be reduced to a MIPS problem as described in Algorithm \ref{alg:ann_comp}. Here, $\mathbf{1}\{ \cdot \}$ represents the indicator function and $\textbf{a} \circ \textbf{b}$ represents the Hadamard product between vectors $\textbf{a}$ and $\textbf{b}$. The number of points in the search space of the above MIPS problem is $N = |\cS|$. The key benefit of reducing the comparison to solving a MIPS instance is that MIPS instances can be solved very efficiently in practice~\citep{neyshabur2015symmetric}, either exactly or approximately. We discuss the time complexity of \ann \ in Section~\ref{subsec:time-complexity}.

\subsection{ Second Algorithm: \aheu }\label{subsec:aheu}

We can extend \ann \ by allowing the use of any sub-routine which solves the MIPS problem. In particular, approaches which solve MIPS approximately and in a randomized manner can also be used to obtain approximately optimal solutions. Such approaches achieve runtime performance gains as they are typically faster while trading off accuracy (see Section \ref{sec:experiments}). 

We now give details regarding the use of an LSH based data structure for solving the MIPS problem approximately. This choice of technique is primarily due to LSH's superior performance in high dimensions and better worst-case runtime guarantees. With suitable modifications, Algorithms \ref{alg:aheu_eff_mod} and \ref{alg:alsh} proposed in the next two subsections can also work with any other subroutine that performs approximate MIPS (with probability of error bounded away from $1$). 

\noindent\textit{Approximation in \aheu{}}: Recall that to solve the  MIPS problem introduced in the \ann \ algorithm, we need the nearest point to $\bv$ in the set $\mathbf{Z}(K)$ according to the inner product metric. Now, instead of finding such a nearest point exactly, we can think of computing it approximately (i.e., it will be \emph{near enough}) as well as probabilistically (i.e., the returned point may not be near enough with some probability). Let us denote the above  operation of finding an approximate nearest neighbor as  $\apargmax$. More precisely, $\apargmax\;(q, P)$\  returns a vector from the set $P$.  This is the  approximate  nearest neighbor (i.e. has a high enough inner product among the points in set $P$) to the query vector $q$ with probability $1-\delta$, where $\delta$ is the probability of error. 

We will focus on the approximation aspect (near enough versus nearest) now and address the error probabilities in Section~\ref{subsec:alsh} (the algorithm \alsh{} in Section~\ref{subsec:alsh} will take errors into account and provide formal error guarantees on the revenue of the assortment returned). The proposed algorithm \aheu{} that addresses the approximation aspect is then similar to \ann \, except for two changes: (a) the operation $\bz^{\tilde{S}} = \argmax_{\bz^S \in \mathbf{Z}(K)   }  \bv \cdot \bz^S$ is replaced by $\bz^{\tilde{S}} = \apargmax_{\bz^S \in \mathbf{Z}(K)   }  \bv \cdot \bz^S$, and (b) the approximation in the quality of the solution returned is taken into account in each iteration. 
 
From Lemma~\ref{lemma:nn2mips}, note that if the LSH data-structure returns a Euclidean approximate nearest neighbor $x$ to a query $y$ such that $\|x-y\|_2 \leq (1+\nu)\|x^*-y\|_2$ (where $x^*$ is the nearest neighbor), then the vector $x$ satisfies the following inner product relation: $1 + (1+\nu)^2(x^*\cdot y - 1) \leq x \cdot y \leq x^*\cdot y$. Let the current threshold for comparison be $K$. Further, let $\hat{K} = 1 + (1+\nu)^2(K - 1)$.   Without loss of generality, let $p_1 \leq 1$, which implies that $\hat{K} \leq K$. 
Because the returned solution $x$ is  only approximate, we can only assert that one of the following inequalities is true for any given comparison, giving us the corresponding update for the next iteration:
\begin{itemize}
\item If $\hat{K} \geq x \cdot y$, then $K \geq x^* \cdot y$, allowing update of the upper bound of the the search region  to $K$.
\item If $K \leq x \cdot y$, then $K \leq x^* \cdot y$, allowing update of the lower bound of the the search region  to $K$.
\item If $\hat{K} \leq x \cdot y$ and $K \geq x \cdot y$, then $\hat{K} \leq x^* \cdot y$, allowing update of the lower bound of the the search region  to $\hat{K}$.
\end{itemize}
We give a proof of the update in the first setting in the electronic companion. The updates in the other two settings are straightforward. These conditions and the efficient transformation discussed above are summarized in Algorithm \ref{alg:aheu_eff_mod}.


\begin{figure}
\noindent\begin{minipage}{\textwidth}
\hspace{\dimexpr-\fboxrule-\fboxsep\relax}\fbox{%
\begin{minipage}[t]{.48\textwidth}
\vspace{0pt}
\begin{algorithm}[H]
\caption{\aheu{}}
\label{alg:aheu_eff_mod}
\begin{algorithmic} 
\REQUIRE{ Prices $\{p_i\}_{i=1}^{n}$, tolerance parameter $\epsilon$, approximate nearest neighbor parameter $\nu$ } \\
\STATE{$L_1 = 0, \ U_1 = p_1 , \ t=1,\bp = (p_1, \cdots, p_n)$} 
\STATE{$\bu^S = (u_1, u_2, \cdots u_n) \ \text{ where } \  u_i = \mathbf{1} \{ i \in S\}$}
\STATE{$S^* = \{1 \}, \ \mathbf{\widehat{Z}} = \{ \mathbf{\hat{z}}^S | \hat{\bz}^S = \left( \bp \circ \bu^S, \bu^S \right), S \in \cS \}$}
\WHILE{ $U_j - L_j > \epsilon$} 
\STATE{ $K_j = \frac{L_j + U_j}{2}$} \\
\STATE{$\hat{K}_j = 1 + (1+\nu)^2(K_j - 1)$} \\
\STATE{$\mathbf{\hat{v}_{K_j}} = (v_1, \cdots, v_n,$\\$\quad\quad\quad -v_1K_j,-v_2K_j, \cdots -v_nK_j)$} \\
\STATE{$ \hat{\bz}^{\tilde{S}} = \apargmax \ (\mathbf{\hat{v}_{K_j}}, \mathbf{\hat{Z}})$}
\IF{$\hat{K}_j >  \frac{ \bv \cdot \hat{\bz}^{\tilde{S}}}{v_0}    $}  
\STATE{$L_{j+1} = L_j , U_{j+1} = K_j$} \\
\ELSIF{$K_j \leq  \frac{ \bv \cdot \hat{\bz}^{\tilde{S}}}{v_0}    $}  
\STATE{$L_{j+1} =  K_j, U_{j+1} = U_j$, $S^* = \tilde{S} $}
\ELSE
\STATE{$L_{j+1} = \hat{K}_j , U_{j+1} = U_j, S^* = \tilde{S} $}
\ENDIF 
\ENDWHILE
\RETURN{$ S^*$}
\end{algorithmic}
\end{algorithm}
\end{minipage}
}
\hspace{\dimexpr-\fboxrule-\fboxsep\relax}\fbox{%
\begin{minipage}[t]{.48\textwidth}
\vspace{0pt}
\begin{algorithm}[H]
\caption{\small{Simpler} \aheu{}}
\label{alg:aheu_simpler}
\begin{algorithmic} 
\REQUIRE{ Prices $\{p_i\}_{i=1}^{n}$, tolerance parameter $\epsilon$} \\
\STATE{$L_1 = 0, \ U_1 = p_1 , \ t=1,\bp = (p_1, \cdots, p_n)$} 
\STATE{$\bu^S = (u_1, u_2, \cdots u_n) \ \text{ where } \  u_i = \mathbf{1} \{ i \in S\}$}
\STATE{$S^* = \{1 \}, \ \mathbf{\widehat{Z}} = \{ \mathbf{\hat{z}}^S | \hat{\bz}^S = \left( \bp \circ \bu^S, \bu^S \right), S \in \cS \}$}
\WHILE{ $U_t - L_t > \epsilon$} 
\STATE{ $K = \frac{L_t + U_t}{2}$} \\
\STATE{$\mathbf{\hat{v}_K} = (v_1, \cdots, v_n, -v_1K,-v_2K, \cdots -v_nK)$} \\
\STATE{$ \hat{\bz}^{\tilde{S}} = \apargmax \ (\mathbf{\hat{v}_K}, \mathbf{\hat{Z}})$}
\IF{$K \leq  \frac{ \bv \cdot \hat{\bz}^{\tilde{S}}}{v_0}    $}  
\STATE{$L_{t+1} =  K, U_{t+1} = U_t$, $S^* = \tilde{S} $}
\ELSE 
\STATE{$L_{t+1} = L_t , U_{t+1} = K$} \\
\ENDIF 
\ENDWHILE
\RETURN{$ S^*$}
\vspace{19mm}
\end{algorithmic}
\end{algorithm}
\end{minipage}
}
\end{minipage}
\end{figure}

\subsection{Third Algorithm: \alsh{}}\label{subsec:alsh}

In \ann, we  narrow our search interval based on the result of the MIPS query.  But when a probabilistic  method is used for  solving the MIPS query (such as when $\apargmax$ is implemented using the LSH data structure discussed before),  there is a chance of  narrowing down the search interval to an incorrect range. To address this, we build on the BZ algorithm~\citep{burnashev1974interval} to accommodate the possibility of failure (not being $(1+\nu)$-NN) in the $\apargmax$ operation. To simplify discussion, we will assume that the LSH structure (which solves the $(1+\nu)$-NN problem) has $1+\nu$ small enough such that for every query $y$, the only point $x'$ that satisfies $||x'-y||_2 \leq (1+\nu) \min_{x} ||x-y||_2$ is $\argmin_{x} ||x-y||_2$. This assumption eliminates the approximation error from the result returned by the data structure, and lets us focus on the probabilistic errors. Thus, the probability of  receiving an incorrect answer ($P_e)$ to the comparison query $K \leq  \max_{S \in \cS} f(S, \bv)$ is  the failure probability ($f$) of the LSH data structure.

The key difference between the BZ algorithm and a standard binary search is the choice of the decision threshold  point at which a comparison is made in each round. In binary search, we choose the  mid-point of the current search interval. As we cannot rule out any part of the original search interval when we receive noisy answers, in the BZ algorithm, we maintain a distribution on the value of the optimal revenue. In every iteration, we test if $K \leq \max_{S \in \cS}f(S, \bv)$ where $K$ is the median of the distribution and then update the distribution based on the result of the comparison. For the analysis of the performance of the BZ based algorithm \alsh \, we require that the errors in every iteration are independent of teach other. We can accomplish this in our setting by querying a different LSH data structure in every round. Thus, we create $T$ distinct LSH data structures that solve the approximate nearest neighbor problem, where $T$ is the number of desired iterations (we show that this is logarithmic in the desired accuracy level, and hence is not a significant overhead).

The \alsh \ algorithm, building on the BZ procedure and making use of LSH data structures, is given in  Algorithm~\ref{alg:alsh}. To keep things fairly general, we discuss the setting in which the feasible assortments do not have a compact representation, although a similar algorithm can be designed for the capacity constrained setting (essentially following the discussion in Section~\ref{alg:ann_outline}). 
To reiterate, because we use probabilistic nearest neighbor queries, we cannot guarantee an $\epsilon$-optimal revenue always, but we show below that we can achieve $\epsilon$-optimality with high probability.\\

\noindent\textit{Error Analysis of} \alsh \: We claim that the probability of error (i.e. the revenue given by the algorithm not being $\epsilon$-optimal) for \alsh \  decays exponentially with the number of iterations of the algorithm. 
Let $Q_e=1-P_e$, where $P_e$ is the true error probability in receiving a correct answer in each comparison step. Also, let $\beta  = 1-\alpha$, where $\alpha<0.5$ is an upper bound on $P_e$.

\begin{theorem}\label{thm:BZError}
After $T$ iterations of \alsh \ (Algorithm \ref{alg:alsh}) we have:
$$\footnotesize P(|\hat{\theta}_{T} - \theta^*| > \epsilon) \leq \frac{p_1 -\epsilon}{\epsilon} \ \left(\frac{P_e}{2\alpha} + \frac{Q_e}{2\beta}\right)^{T},$$
where $\theta^*$ is the true optimal revenue. 
\end{theorem}
The proof of this Theorem is presented in the electronic companion. It is based on a similar analysis done by \cite{burnashev1974interval} and \cite{castro2006upper}. We now make a few remarks. The bound above depends on the quantity $P_e$. We do the following modification when only an upper bound on its value (say $P_{max}$) is known (this is the case for LSH based approximate nearest neighbor data structures). Let $W(P_e, \alpha) = \frac{P_e}{2\alpha} + \frac{Q_e}{2\beta}. $ This is a linear and increasing function of  $P_e$. Thus,  
$W(P_e, \alpha) \leq W(P_{max}, \alpha)$. By appropriately tuning the LSH data structure (this is always possible, See Section~\ref{subsec:LSH}), set $P_{max} < 0.25$. Choosing $\alpha = \sqrt{P_{max}}$ (a valid upper bound on $P_e$), we get $W(P_{max}, \sqrt{P_{max}}) = 0.5 + \sqrt{P_{max}}$. Thus, we have
$ P(|\hat{\theta}_{T} - \theta^*| > \epsilon) \leq \frac{p_1 -\epsilon}{\epsilon} \ \left(\sqrt{P_{max}} + 0.5 \right)^{T}. $
With the above reduction,  the number of iterations  required to get to a desired reliability level can now be estimated. If the desired accuracy level is $\gamma$ i.e., $ P(|\hat{\theta}_{T} - \theta^*| > \epsilon) \leq \gamma$, then $ T \geq \log_{0.5 + \sqrt{P_{max}}} \frac{\gamma \epsilon}{p_1 -\epsilon}$. Thus, the number of iterations grows logarithmically in the desired accuracy level. Finally, note that $\hat{\theta}_T$ can be greater than the optimal revenue, and \alsh \ does not always output an assortment that has revenue greater than or equal to $\hat{\theta}_T$.

\begin{algorithm}
\caption{\alsh{}}
\label{alg:alsh}
\footnotesize
\begin{algorithmic} 
\REQUIRE{ Prices $\{p_i\}_{i=1}^{n}$, tolerance parameter $\epsilon$ such that $p_1 \epsilon^{-1} \in \mathbb{N}$}, number of steps $T$, upper bound $\alpha < 0.5$ on the error probability in the $\apargmax$ operation, and let $\beta := 1 - \alpha$. \\
Posterior $\pi_j : [0,p_1] \rightarrow  \mathbb{R} $ after $j$ stages:  
$$ \pi_j(x) = \sum_{i=1}^{p_1 \epsilon ^{-1}}a_i(j)\mathbf{1}_{I_i}(x), $$
where  $I_1 = [0,\epsilon]$ and $ I_i = (\epsilon (i-1),\epsilon i]$ for  $i \in \{2, \cdots, p_1 \epsilon^{-1} \}$. Let $\ba=[a_1(j), \cdots, a_{p_1 \epsilon^{-1}(j)} ]$. \\
Initialize $a_i(0) = {p_1}^{-1} \epsilon \ \forall i , \hat{S} = \{ 1\},  j=0$. 

\WHILE{ $j < T$} 
\STATE{ \textbf{Sample Selection:} Define $u(j)$ such that $\sum_{i=1}^{u(j)-1}a_i(j) \leq \frac{1}{2} \ , \sum_{i=1}^{u(j)}a_i(j) > \frac{1}{2} $}. Let \\
$K_{j+1} = 
\begin{cases}
{p_1}^{-1} \epsilon (u(j)-1)  & \mbox{ with probability }  Q(j), \textrm{ and} \\
{p_1}^{-1} \epsilon u(j) & \mbox{ with probability }  1 - Q(j), \\
\end{cases} 
$ where $Q(j) = \frac{\tau_2(j)}{\tau_1(j)+\tau_2(j)}$, and  \\
$\tau_1(j) = \sum_{i=u(j)}^{p_1\epsilon^{-1}}a_i(j) - \sum_{i=1}^{u(j)-1}a_i(j)$, \\
$\tau_2(j) = \sum_{i=1}^{u(j)}a_i(j) - \sum_{i=u(j)+1}^{p_1 \epsilon ^{-1}}a_i(j)$. 

\vspace{3mm}
\textbf{Noisy Observation:} \STATE{$ \hat{\bz}^{\tilde{S}} = \apargmax \ (\mathbf{\hat{v}}_{K_{j+1}}, \mathbf{\hat{Z}})$ } 
\STATE{If $f(\tilde{S}, v) > f(\hat{S}, v), \textrm{ set } \hat{S} = \tilde{S}$}.

\vspace{3mm}
\textbf{Update posterior:} The posterior is updated using the Bayes rule. Note that $K_{j+1} = p_1^{-1} \epsilon u, u \in \mathbb{N}.$ Define 
$$  h(K_{j+1}) = \mathbf{1}\{K_{j+1} \leq  \frac{\mathbf{\hat{v}}_{K_{j+1}} \cdot \hat{\bz}^{\tilde{S}} }{v_0} \}, \quad\textrm{ and }\quad \tau = \sum_{i=1}^u a_i(j) - \sum_{u+1}^{p_1 \epsilon^{-1}} a_i(j). $$

For  $i \leq u$, we have the update \\
{$a_i(j+1) = \begin{cases} 
\frac{2\beta}{1+\tau(\beta - \alpha)} & \mbox{if }  h(K_{j+1}) = 0, \textrm{ and} \\
\frac{2\alpha}{1-\tau(\beta - \alpha)} & \mbox{if }  h(K_{j+1}) = 1. \\
\end{cases}
$}

For  $i > u$, we have the update \\
{$a_i(j+1) = \begin{cases} 
\frac{2\alpha}{1+\tau(\beta - \alpha)} & \mbox{if }  h(K_{j+1}) = 0, \textrm{ and} \\
\frac{2\beta}{1-\tau(\beta - \alpha)} & \mbox{if }  h(K_{j+1}) = 1. \\
\end{cases}
$}
\ENDWHILE
\vspace{3mm}
\STATE{ $\theta_T$ is defined as the median of the posterior distribution i.e. $ \int_0^{\theta_T }\pi_T(x) = \frac{1}{2} $. }

\RETURN{$ \hat{S},$ and $ \hat{\theta}_T =  \max \left( \theta_T, f(\hat{S},\bv) \right)$} 
\end{algorithmic}
\end{algorithm}

\subsection{Time and Space Complexities of the Proposed Algorithms}\label{subsec:time-complexity}

In \ann,  the MIPS query in every iteration can be solved exactly in time $O(\noofprod \noofset)$ when there is no compact representation for the set of feasible assortments. Here $\noofprod$ \ is the dimension of the vectors (equals the number of items) and $\noofset$ is the number of points in the search space (equals the number of feasible sets).  In every iteration, we cut down the search space for the optimal revenue by half. We start with a search space of range $[0,p_1]$. Thus, the number of iterations to get to the desired tolerance is $ \left\lceil \log \frac{p_1}{\epsilon} \right \rceil$. Thus, the time taken by \ann \ is $O(\noofprod \noofset \log \frac{p_1}{\epsilon}  )$. While this may seem worse than linear scan (whose complexity is $\Theta(nN)$), exact nearest neighbor searches in vector spaces are very efficient in practice, and lead to \ann \ having a much better computational performance compared to linear scan (see Section~\ref{sec:experiments}).  
Both \ann{} and the exhaustive search typically don't require any additional storage except the space needed for storing all the feasible sets. Thus, their space complexity is $O(nN)$. For the case of assortment planning under capacity constraint, the  \textsc{compare-step} amounts to finding top $C$ among $n$ elements. This has a complexity of $O(n\log C)$. Thus, the time taken by the algorithm is  $O(n\log C \log \frac{p_1}{\epsilon})$. Note that other algorithms  for solving the capacitated assortment planning problem like ADXOpt and \textsc{STATIC-MNL} have time complexity quadratic in $\noofprod$. Exhaustive search is not competitive here (its time complexity is O($n^C$)).

In \aheu, it takes $O(\noofprod \noofset ^ \rho)$ time to run the approximate near neighbor query, with $\rho < 1$ (See Section~\ref{subsec:LSH}). The number of iterations is $\left\lceil \log \frac{p_1}{\epsilon - 2(\nu^2 + 2\nu)} \right \rceil$ (see proof in the electronic companion). Thus, the time complexity is $O\left(\noofprod \noofset^ \rho  \log \frac{p_1}{\epsilon - 2(\nu^2 + 2\nu)} \right)$. The space complexity  is $O( \noofprod \noofset^ {(1+\rho)} )$. Similarly, the time complexity of \alsh{} \ is $O\left(\noofprod \noofset^ {\rho}\log \frac{p_1 - \epsilon}{\gamma\epsilon}\right)$ and the space complexity is $O(\noofprod \noofset^ {(1+\rho)} )$.

\section{Experiments}
\label{sec:experiments}

We empirically validate the runtime performance of two of the proposed algorithms, namely \ann \ and \aheu \ using real and synthetic datasets. Firstly, because we expect \alsh \ to perform very similar to \aheu \ in terms of accuracy, we omit its performance comparisons. Second, we implement a simpler version of \aheu \, as shown in Algorithm~\ref{alg:aheu_simpler}, where instead of comparing with two thresholds as described in Section~\ref{subsec:aheu}, we compare with a single comparison threshold akin to \ann. This is because the approximation parameter $\nu$ depends on the implementation of $\apargmax$, with different implementations giving different empirical results. Further, the implementation we choose below implicitly changes the approximation parameter in a data-driven way to get the best performance for inner-product search. Hence, Algorithm~\ref{alg:aheu_simpler} is defined to side-step these complexities. A consequence of this is that the reported performance numbers may be slightly better in terms of time complexity and slightly worse in terms of solution quality and percentage errors in revenue. Nonetheless, as we show below, the redefined simpler algorithm (w.l.o.g. we will call it \aheu{} in this Section) does capture the key aspects that we hoped for: better performance in terms of computation time while working with arbitrary collection of feasible assortments (as well as the capacitated setting) without much impact on revenue or solution quality. For reproducibility, code corresponding to all the experimental results presented in this Section is provided at the following link: \url{https://github.com/thejat/data-driven-assortments}.

For the case of general assortments, we compare our algorithms against an algorithm which performs exhaustive search. And for the case of cardinality constrained assortments, we compare these with other algorithms such as \textsc{Static-MNL}~\citep{capMNL}, ADXOpt~\citep{jagabathula2014assortment} and a linear programming (LP) formulation~\citep{davis2013assortment}. All experiments are run on a 6 core 64GB 64-bit intel machine (i7-6850K \@ 3.6GHz) with python 2.7. We use LSH Forest~\citep{scikit-learn,andoni2017lsh} and NearestNeighbors from Scikit-learn for solving MIPS approximately and exactly respectively, and  CPLEX 12.7 for solving the LP. Performance is measured in terms of the mean computational time (sec), mean relative error in the revenue obtained as well as the mean overlap between the assortment output by our algorithms and the optimal assortment output, where the average is over multiple Monte Carlo runs.

\subsection{Datasets} 
We use two different types of real data sets, one as source for real prices and the second as a source for general data-driven assortments without a compact representation. For prices, we use the publicly available online micro price dataset from the Billion Prices Project~\citep{IAH6Z6_2016} to generate item prices. This dataset contains daily prices for all goods sold by 7 large retailers in Latin America (3 retailers) and the USA (4 retailers) between 2007 to 2010. Among the US retailers, we use pricing data from a supermarket and an electronics retailer to generate our assortment planning instances of varying sizes. The former contains 10 million daily observations for 94,000 items and the latter contains 5 million daily observations for 30,000 items. We use prices from 50 different days when generating  instances for 50 Monte Carlo runs under different settings described below. The collection of assortments in these instances are either general or capacitated. The MNL parameters  are each chosen from the uniform distribution $U[0,1]$.

For generating data-driven assortments, we use publicly available transaction datasets used in frequent itemset mining~\citep{borgelt2012frequent}. In particular, we use the retail, foodmart, chainstore and e-commerce transaction logs~\citep{fournier2014spmf} to create collections of general assortments. We use the FPgrowth algorithm from the spmf~\citep{fournier2014spmf} program to first generate frequent itemsets with appropriate minimum supports and then prune out frequent itemsets with low cardinalities (e.g., singletons) to obtain our collection of assortments. Table~\ref{table:freq-itemset-stat} describes some statistics of the assortments generated. We again create 50 instances from each dataset by generating the price and the MNL parameter vectors using uniform distributions $U[0,1000]$ and $U[0,1]$ respectively.

\subsection{General Assortments}

In the first setting, we use the assortments that were obtained using frequent itemset mining and post-processing (to remove assortments of low cardinality) in order to compare the performance of \ann, \aheu \ and exhaustive search. The tolerance parameter $\epsilon$ is set to $0.1$ for \ann \ as well as \aheu  \ in this and all subsequent experiments. In the LSH Forest~\citep{scikit-learn} subroutine, the accuracy and the computational efficiency is controlled by two parameters: (a) number of candidates (default value set to $80$), and (b) number of estimators (default value set to $20$). Note that the LSH forest structure contains multiple LSH trees. The number of estimators parameter specifies the number of trees and the number of candidates specifies the minimum number of near neighbors that should be chosen from each tree. 
The results are plotted in Figure~\ref{fig:gen-ast-real} (mean values across 50 runs are reported). As can be inferred from the plots, \aheu \ is $2$-$3\times$ faster than exhaustive search and also better than \ann \ without sacrificing much of the revenue. As mentioned before, these instances cannot be efficiently handled by either integer programming formulations (unless there is an efficient representation of the feasible assortments) or by other specialized approaches. 

In the second setting, we use the pricing data from the Billion Prices project and generate a fixed number of assortments from the set of all assortments uniformly at random. In particular, we vary the number of assortments to be from the set $\{100,200,400,800,1600,3200,6400, \allowbreak 12800, 25600, 51200\}$.  Again, we report results averaged over 50 Monte Carlo runs, as shown in Figure~\ref{fig:gen-ast-real-price}. We run multiple versions of \aheu \ with different accuracies controlled by the number of candidates parameter and the number of estimates parameter ranging over sets $\{80,160,200\}$ and 
$\{20,40,100\}$ respectively for the underlying LSH Forest~\citep{scikit-learn} subroutine. As can be observed, our proposed algorithms are much better than exhaustive search in terms of processing time while being very close to the optimal in terms of solution quality and revenue. In particular, we can trade off speed (timing performance) of \aheu \ with accuracy (lower percentage of assortment set overlap). We obtained qualitatively similar performances for synthetic data (prices uniformly generated from the interval $[0,1000]$), and thus, omit their plot here.

\subsection{Cardinality-constrained Assortments}

In this experiment, we explore how \ann \ fares as compared to the specialized algorithms viz., ADXOpt, LP and Static-MNL in the capacity-constrained setting (exhaustive search is not considered because its performance is an order of magnitude worse than all these methods).  We run experiments with both synthetically generated instances as well as instances generated with prices from the Billion Prices dataset (to capture real price distributions), and observe similar qualitative results, so we omit the former. The cardinality parameter is chosen to be $50$ in all cases. Figure~\ref{fig:cap-real-price-prod} shows the performance of all the algorithms for the Billion Prices dataset. The number of items was varied from the set $\{100, 250, 500, 1000, 3000, 5000, 7000,10000,\allowbreak 20000\}$. We  observe that \ann \ runs much faster than both Static-MNL and LP. For instance, when the number of items is $20000$, \ann \ computes solutions in time that is two orders of magnitude less compared to ADXOpt and one order of magnitude less  compared to LP on average (we discount the time needed to set up the LP instance and only report the time to solve the instance using the CPLEX 12.7 solver, otherwise the relative gains would be much higher). At the same time, the percentage loss in revenue of assortments reported by \ann \  is very close to 0. 
We have re-plotted the timing performance of ADXOpt separately in Figure~\ref{fig:cap-real-price-prod-adxopt-static-mnl}, as its running time (plotted as intensity) varies as a function of the size of the optimal assortment, which is not the case with \ann \ and LP. If the instance happens to have a small optimal assortment, ADXOpt can get to this solution very quickly. On the other hand, it spends a lot of time when the optimal assortment size is large. This is illustrated through an intensity plot as opposed to a one-dimensional mean timing plot. Similarly, we plot the timing performance of Static-MNL separately in Figure~\ref{fig:cap-real-price-prod-adxopt-static-mnl} because its performance turned out to be much worse than all the other methods even for moderate sized instances.

In summary, our algorithm \ann \ is competitive with the state of the art algorithms, viz., ADXOpt, Static-MNL and LP in the capacitated setting and both \ann \ and \aheu are scalable in the general data-driven setting. Further, these algorithms vastly increase the assortment planning problem instances that can be solved efficiently under the MNL choice model. For instance, we show computational results when instances have a number of items that is of the order of $\sim10^5$ easily, whereas the regime in which experiments of the current state of the art methods~\citep{Vel} were carried out is with $\noofprod \sim 10^3$, thus representing two orders of magnitude improvement. By trading off accuracy, our algorithms can also achieve 10s-100s of millisecond budgets discussed in Section~\ref{sec:introduction}.
\section{Concluding Remarks}
\label{sec:conclude}

We proposed multiple efficient algorithms that solve the assortment optimization problem under the Multinomial Logit (MNL) purchase model even when the feasible assortments cannot be compactly represented. In particular, we motivate how frequent itemsets can be used as candidate assortments, making the planning problem data-driven. Our algorithms are iterative and build on binary search and fast methods for maximum inner product search to find optimal solutions for large scale instances. Though solving large scale instances efficiently under flexible assortments is a significant gain, there is scope for extending this work in many ways. For instance, studies by psychologists have revealed that buyers are affected by the assortment size as well as how frequently they change over the course of their interactions~\citep{iyengar2000choice}. Being able to handle arbitrary sets of assortments which can be based on frequent itemset mining, it would be interesting to model the purchase of multiple items and perform assortment optimization in that setting.

Our algorithms are meaningful both in online and offline setups: in online setups such as e-commerce applications, our algorithms scale as the problem instances grow. In offline setups, our algorithms afford flexibility to the decision maker  by allowing optimization over arbitrary assortments, which could be driven by business insights or transaction log mining.

\bibliographystyle{informs2014} 
\bibliography{assort_lsh} 

\section*{Tables and Figures}
\label{sec:table}
\begin{table}[h]
\begin{center}

\caption{Assortments generated from transactional datasets. \label{table:freq-itemset-stat}}
\resizebox{.6\columnwidth}{!}{%
\begin{tabular}{ |l|c|c|c|c| } 
 \hline
 Dataset & Retail & Foodmart & Chainstore & E-commerce \\ \hline
 Number of transactions & 88162 &4141 & 1112949 & 540455 \\ \hline
 Number of items & 3160 & 1559 & 321 & 2208 \\ \hline
 Number of general assortments & 80524 & 81274 & 75853 & 23276 \\ \hline
 Size of largest assortment & 12 & 14 & 16 & 8\\ \hline
 Size of smallest assortment & 3 & 4 & 5 & 3\\ \hline
\end{tabular}
}%
\end{center}
\end{table}
\begin{figure}[ht]
 \centering
  \includegraphics[width=.3\textwidth]{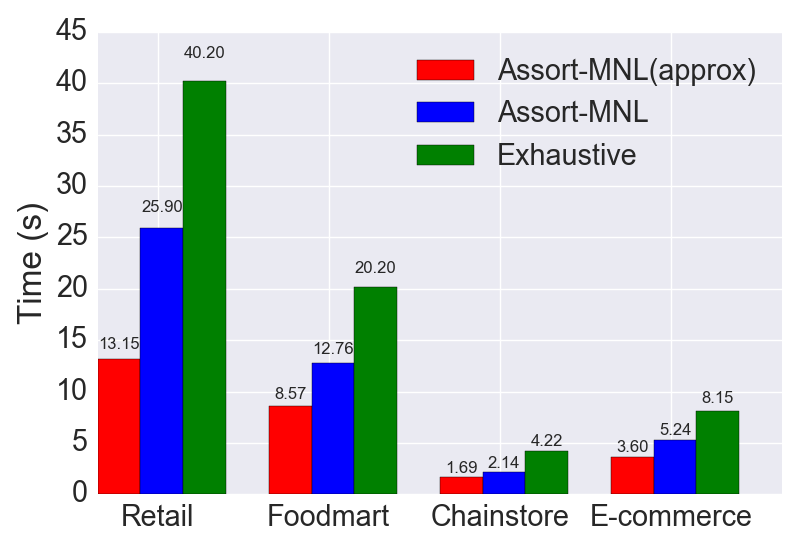}
  \includegraphics[width=.3\textwidth]{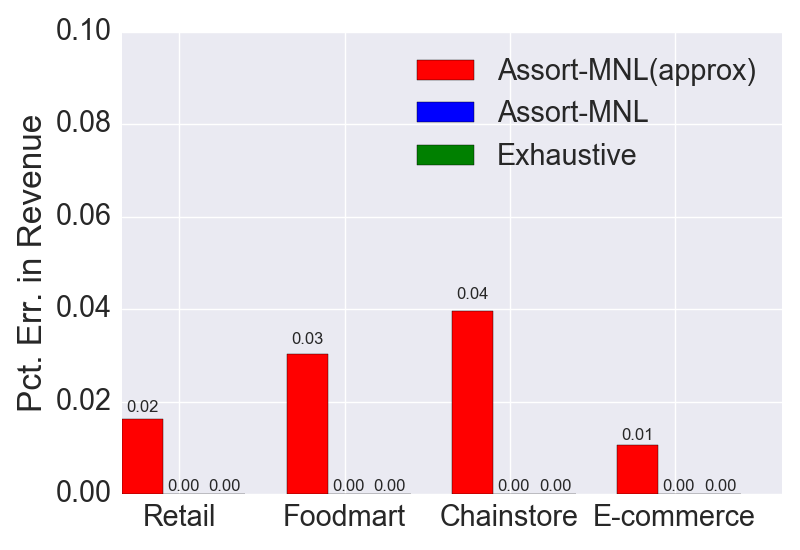}
 \includegraphics[width=.3\textwidth]{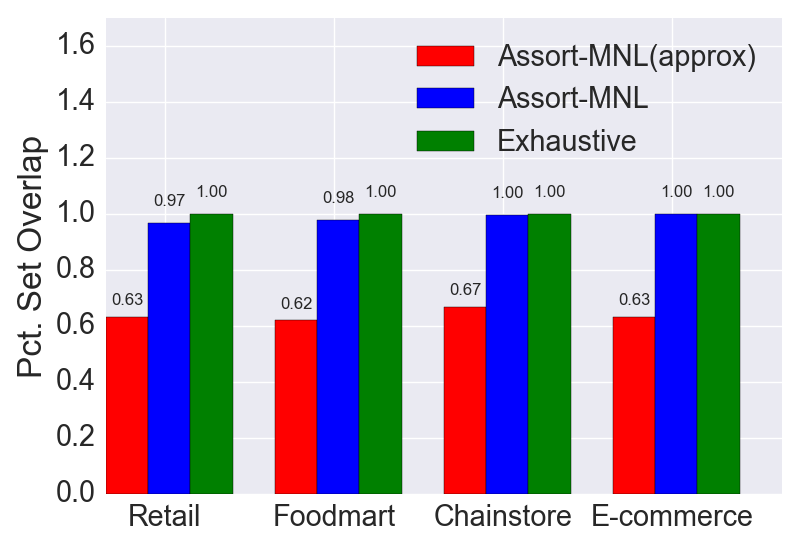}
 \caption{Performance of \ann{} and \aheu{} over general data-driven instances derived from four different frequent itemset datasets.  \label{fig:gen-ast-real}}
 \end{figure}
 \begin{figure}[ht]
 \centering
  \includegraphics[width=.3\textwidth]{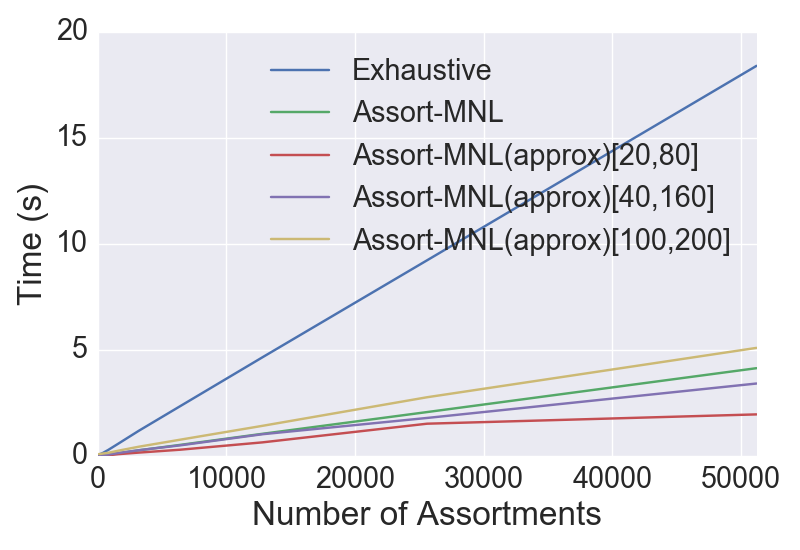}
  \includegraphics[width=.3\textwidth]{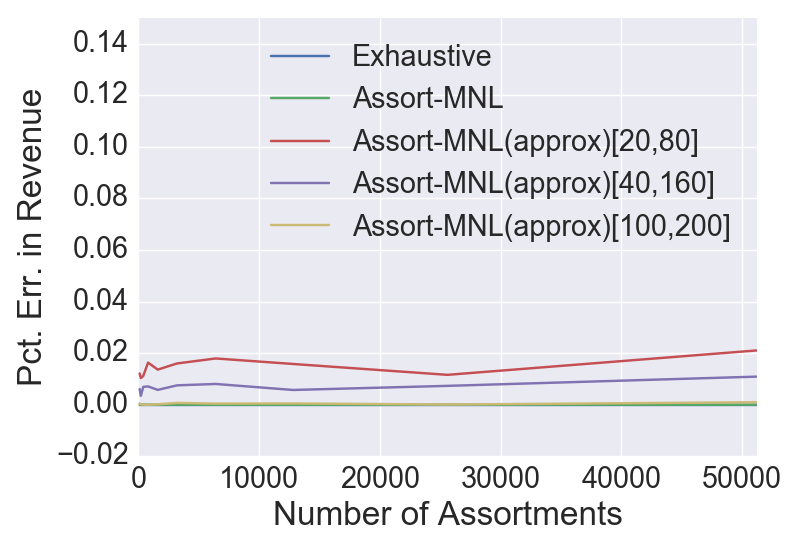}
 \includegraphics[width=.3\textwidth]{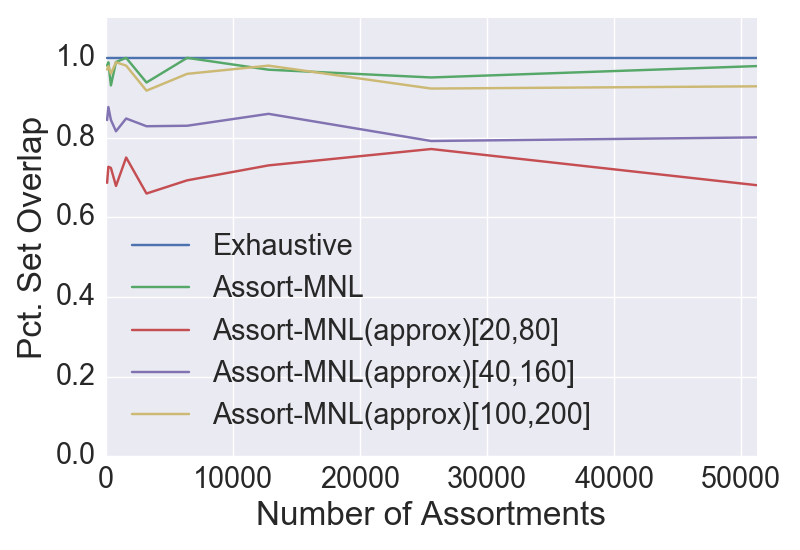}
 \caption{Performance of \ann{} and \aheu{} over general data-driven instances derived from the Billion prices dataset. The x-axis corresponds to the number of feasible assortments (uniformly sampled).
 \label{fig:gen-ast-real-price}}
 \end{figure} 
\begin{figure}[ht]
 \centering
  \includegraphics[width=.3\textwidth]{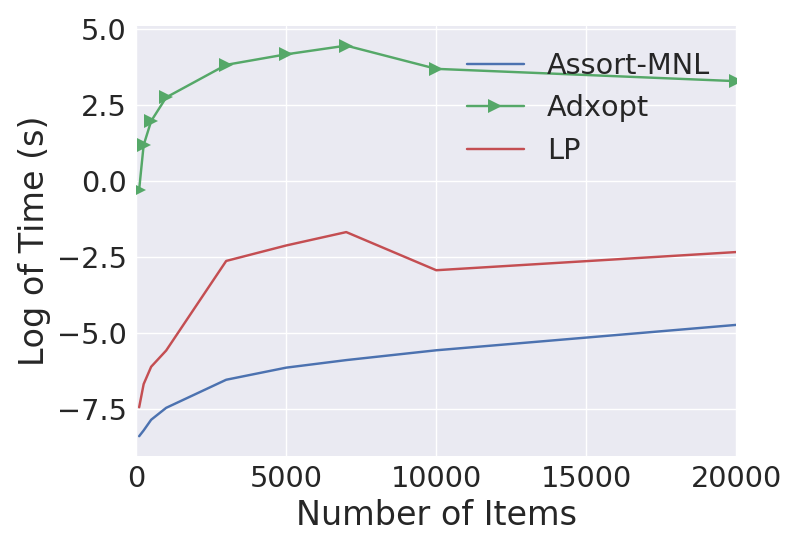}
  \includegraphics[width=.3\textwidth]{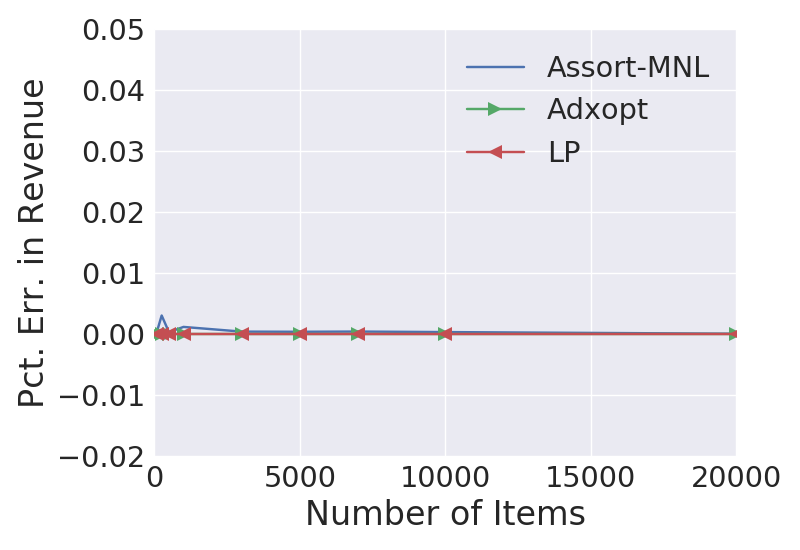}
 \includegraphics[width=.3\textwidth]{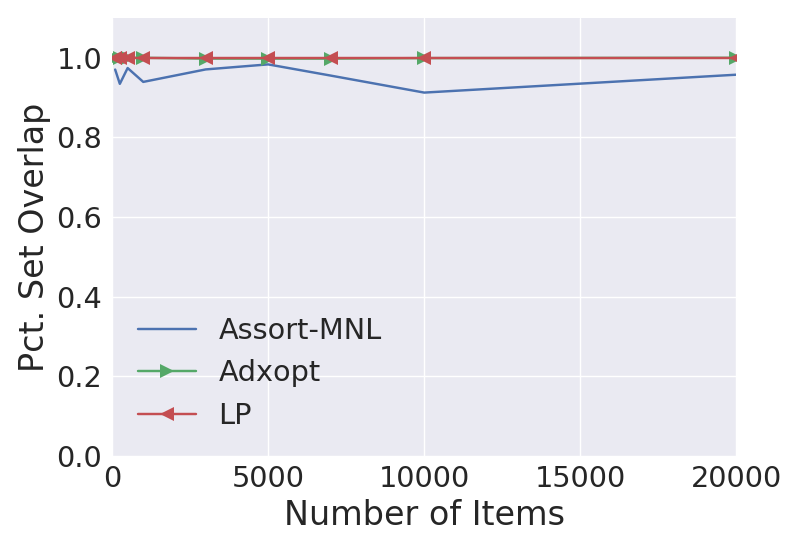}
 \caption{Performance of \ann{} and \aheu{} in the capacitated setting, over instances derived from the Billion prices dataset. The x-axis corresponds to the number of items.
 \label{fig:cap-real-price-prod}}
 \end{figure} 
 \begin{figure}[ht]
 \centering
  \includegraphics[width=.3\textwidth]{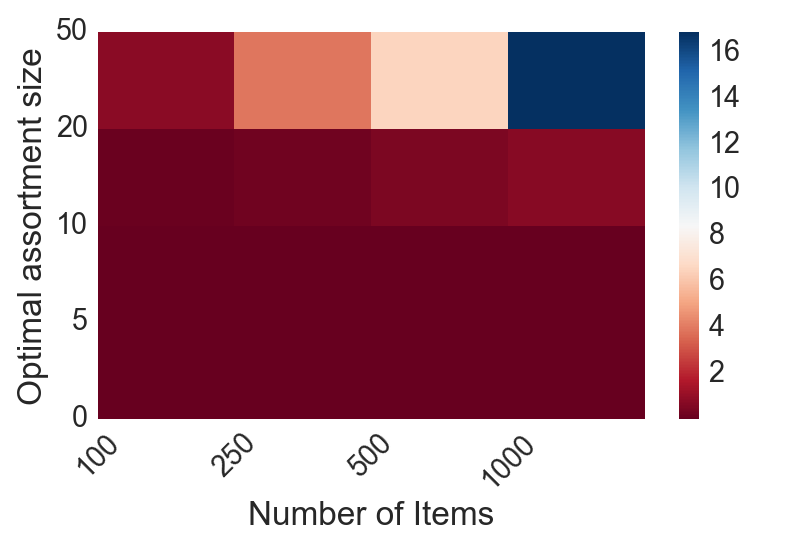}
  \includegraphics[width=.3\textwidth]{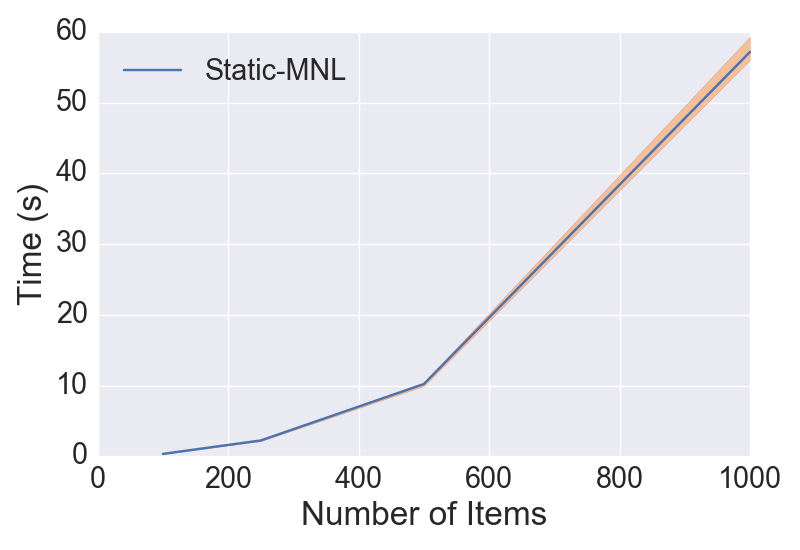}
 \caption{Performance of ADXOpt (left) and Static-MNL (right) over instances derived from the Billion prices dataset. The intensity (time in seconds) shows that the performance of ADXOpt is highly dependent on the size of the optimal assortment. Static-MNL's timing performances were found to be an order of magnitude worse even for moderate sized instances.
 \label{fig:cap-real-price-prod-adxopt-static-mnl}}
 \end{figure} 


\section*{Appendix}


\noindent\textbf{I: An Illustration of how to compute nearest neighbor using near neighbor data structures}

To give intuition about how near neighbor data structures can be used to compute the nearest neighbor, we describe a simple strategy, which is as follows. We create multiple near neighbor data structures as described in Theorem~\ref{thm:near-neighbor} using different threshold values ($r$) but with the same success probability $1-f$ (by amplification for instance). When a query vector is received, we calculate the near-neighbors using the hash structure with the lowest threshold. We continue checking with increasing value of thresholds till we find at least one near neighbor. Let $\widetilde{r}$ be the first threshold for which there is at least one near neighbor. This implies that the probability that we don't find the true nearest neighbor is at most $f$ because the near neighbor data structure with the threshold $\widetilde{r}$ has success probability $1-f$.  Assuming that the different radii in the data structures are such that the number of points returned for the threshold  $\widetilde{r}$ is sublinear in $N$ i.e. $O(N^\eta)$ for some $\eta < 1$, gives us the desired data structure for the nearest neighbor problem. While the last assumption is reasonable as long as the different radii are not very far apart, the structure corresponding to Theorem~\ref{thm:nearest-neighbor} does not need it. Further details on this data structure can be found in ~\cite{har2012approximate}.\\

\noindent\textbf{II: Solving MIPS using LSH }

We illustrate how the MIPS problem can be solved approximately using a specific LSH family, as described in~\cite{neyshabur2015symmetric}. For $x \in \mathcal{R}^{\noofprod}, || x ||_2 \leq 1$, we first define a preprocessing transformation $T: \mathcal{R}^{\noofprod} \rightarrow \mathcal{R}^{\noofprod+1}$ as $ T(x) =  [x; \sqrt{1- ||x||_2^2}]$. We sample a spherical random vector $a \sim \mathcal{N}(0,I)$ and define the hash function as $h_a(x) = sign(a\cdot x)$. All points in the search space are preprocessed as per the transformation $T(\cdot)$ and then hashed using hash functions, such as the above, to get a set of indexed points.
During run-time when we obtain the query vector, it is also processed in the same way i.e. first transformed through $T(\cdot)$ and then 
through the same hash functions that were used before. Checking for collisions in a way that is similar to the near neighbor problem in Section~\ref{subsec:LSH} retrieves the desired similar points.

Without loss of generality, assume that the query vector $y$ has $||y||_2 = 1$. The following guarantee can be shown for the probability of collision of two hashes:
$$\mathcal{P}[h_a(T(x)) = h_a(T(y))] = 1 - \frac{cos^{-1}(x\cdot y)}{\pi},$$ 
which is a  decreasing function of the inner product $x\cdot y$. For any chosen threshold value $D$ and  $c<1$, we consequently get the following:
\begin{itemize}
\item  If  $x\cdot y \geq D$, then
$\mathcal{P}[h_a(T(x)) = h_a(T(y))]  \geq 1 - \frac{cos^{-1}(D)}{\pi}$.
\item  If  $x\cdot y \leq cD$, then
$\mathcal{P}[h_a(T(x)) = h_a(T(y))]  \leq 1 - \frac{cos^{-1}(cD)}{\pi}$.
\end{itemize}

We are by no means restricted to using the above hash family. Because of the transformation $T(\cdot)$, we have obtained a nearest neighbor problem that is equivalent to the original inner product search problem. Thus, any hash family that is appropriate for the Euclidean nearest neighbor problem can be used (see Section~\ref{subsec:LSH}). \\

\noindent\textbf{III: Justifying the updates in \aheu \ }

\begin{proposition}
If $\hat{K} \geq x \cdot y$, then $K \geq x^* \cdot y$.
\end{proposition}
\begin{proof} \emph{Proof.}
We will prove the contrapositive of this statement. That is, we will show that if $x^*$ is such that $x^*\cdot y$ is indeed greater than $K$, then $\hat{K} \leq x \cdot y$.  If we could solve the exact MIPS problem with the query vector $y$ (that depends on $K$), we would get the solution $x^*$ and immediately conclude that $x^*\cdot y \leq K$. In reality, we are returned an approximate solution $x$ with the guarantee that $1 + (1+\nu)^2(x^*\cdot y - 1) \leq x \cdot y \leq x^*\cdot y$. But since, $K \leq x^*\cdot y$, we also have $1 + (1+\nu)^2(K - 1) \leq 1 + (1+\nu)^2(x^*\cdot y - 1)$. This readily implies $\hat{K} \leq x \cdot y$. \qed
\end{proof}

\noindent\textbf{IV: Proof of Theorem \ref{thm:BZError} }

\proof{}
The proof from \cite{burnashev1974interval} (that quantifies the error in the vanilla BZ algorithm) has been modified for our setting in the following manner: (a) we remove the restriction that the noise distribution in the $\apargmax$ operation should be Bernoulli, and (b) we generalize the setting such that the error probability $p_j$ in the $\apargmax$ operation at every iteration $j$ can be different with $P_e = \max_{j \in {1, \cdots T}} p_j$.

The interval $[0,p_1]$ is divided into subintervals of width $\epsilon, \ a_i(j)$ denotes the posterior probability that the
optimal revenue $\theta^*$ is located in the $i$-th subinterval after the $j$-th iteration, and $\theta_j$ denotes the median of the posterior after  the $j$-th iteration. Let $Y_j = h(K_j)$ denote the outcome of the comparison. 
Let $\theta^*$ be fixed but arbitrary, and define $u(\theta^*)$ to be the index of the bin $I_i$ containing $\theta^*$, that is $\theta^* \in I_{u(\theta^*)}$. In general, let $u(j)$ be defined as the index of the bin containing the median of the posterior distribution in the $j$-th iteration.
We define two more functions of $\theta^*$, $M_{\theta^*}(j)$ and $N_{\theta^*}(j)$ as below - 

\begin{gather*}
M_{\theta^*}(j) = \frac{1- a_{u(\theta^*)}(j)}{a_{u(\theta^*)}(j)}, \text{ and} \\
N_{\theta^*}(j + 1) = \frac{M_{\theta^*}(j + 1)}{M_{\theta^*}(j)} =\frac{ a_{u(\theta^*)}(j)(1 -  a_{u(\theta^*)}(j + 1))} {a_{u(\theta^*)}(j + 1)(1 -  a_{u(\theta^*)}(j))}.
\end{gather*}


After $T$ observations our estimate of $\theta^*$ is the median of the posterior density $\pi_T(x)$, which means that $\theta_T \in I_{u(T)}$. Taking this into account we conclude that

\begin{align*}
P(|\theta_T - \theta^*| > \epsilon) \ & \leq \ P(a_{u(\theta^*)}(j) < 1/2), \\
& = \ P(M_{\theta^*}(T) > 1), \\
& \leq \ E[M_{\theta^*}(T)]. \hspace{2mm} (\text{because of Markov's inequality}).
\end{align*}


Using the definition of $N_{\theta^*}(j)$, and manipulating conditional expectations we get:
\begin{align*}
E[M_{\theta^*}(T)] & =  E[M_{\theta^*}(T - 1)N_{\theta^*}(T)], \\
& = E [E[M_{\theta^*}(T - 1)N_{\theta^*}(T)|\ba(T - 1)]], \\
& = E [M_{\theta^*}(T - 1)E [N_{\theta^*}(T)|\ba(T - 1)]], \\
& \vdots \\
& = M_{\theta^*}(0)E [E[N_{\theta^*}(1)|\ba(0)] \cdots E[N_{\theta^*}(T)|\ba(T - 1)]], \\
&\leq M_{\theta^*}(0) \left\lbrace \max_{j \in \lbrace0,1,\cdots T-1 \rbrace}\max_{\ba(j)}E[N_{\theta^*}(j+1)|\ba(j)] \right\rbrace ^n.
\end{align*}

The error in the $\apargmax$ operation is not Bernoulli but has the following behaviour. If $K_{j} > \theta^*$, then $Y_j= 0$ (there is no error). If $K_{j} < \theta^*$, then an error can occur and 
\begin{center}
$h(K_{j}) = 
\begin{cases}
1 & \text{with probability } 1-p_j, \textrm{ and} \\
0 & \text{with probability }p_j. \\
\end{cases}
$ \\
\end{center}

To bound $P(|\hat{\theta}_T - \theta^*| > \epsilon)$, we are going to consider three cases: (i) $u(j) = u(\theta^*)$; (ii) $u(j) > u(\theta^*)$; and (iii) $u(j) < u(\theta^*)$. For each of these cases, we first derive an expression for $N_{\theta^*}(j + 1)$ (with some algebraic manipulation and simplification) as follows:


$N_{\theta^*}(j+1) = 
\begin{cases}
\frac{1+(\beta-\alpha)x}{2\beta} \text{ with probability } B=1-A. \textrm{ and} \\ 
\frac{1-(\beta-\alpha)x}{2\alpha} \text{ with probability } A. \\ 
\end{cases}
$ \\
Thus:
\begin{enumerate}[(i)]
\item When $u(j) = u(\theta^*)$ and
  \begin{enumerate}
  \item $K_{j+1} = p_1^{-1}\epsilon (u(j)-1) \text{, then } 
   x = \frac{\tau_1(j) - a_{u(\theta^*)}(j)}{1- a_{u(\theta^*)}(j)}, \textrm{ and } A = p_{j+1}$.
   \item $K_{j+1} = p_1^{-1}\epsilon u(j) \text{, then } 
   x = \frac{\tau_2(j) - a_{u(\theta^*)}(j)}{1- a_{u(\theta^*)}(j)}, \textrm{ and }  A = 0$.
  \end{enumerate}
  \item When $u(j) > u(\theta^*)$ and
  \begin{enumerate}
  \item $K_{j+1} = p_1^{-1}\epsilon (u(j)-1) \text{, then } 
     x =  - \frac{\tau_1(j) + a_{u(\theta^*)}(j)}{1- a_{u(\theta^*)}(j)}, \textrm{ and }  A = 0$.
     \item $K_{j+1} = p_1^{-1}\epsilon u(j) \text{, then } 
     x = \frac{\tau_2(j) - a_{u(\theta^*)}(j)}{1- a_{u(\theta^*)}(j)}, \textrm{ and }  A = 0$.
  \end{enumerate}
   \item When $u(j) < u(\theta^*)$ and
  \begin{enumerate}
  \item $K_{j+1} = p_1^{-1}\epsilon (u(j)-1) \text{, then } 
       x =   \frac{\tau_1(j) - a_{u(\theta^*)}(j)}{1- a_{u(\theta^*)}(j)}, \textrm{ and }  A = p_{j+1}$.
       \item $K_{j+1} = p_1^{-1}\epsilon u(j) \text{, then } 
       x = - \frac{\tau_2(j) + a_{u(\theta^*)}(j)}{1- a_{u(\theta^*)}(j)}, \textrm{ and }  A = p_{j+1}$.
    \end{enumerate}
\end{enumerate}

As $0\leq \tau_1(j) \leq 1$ and $0 < \tau_2(j) \leq 1$, we have $|x| \leq 1$ in all the above cases. Define 
\begin{align*}
g_A(x) & = \frac{B(1+(\beta- \alpha)x)}{2\beta} + \frac{A(1-(\beta-\alpha)x)}{2\alpha} \\
& = \frac{B}{2\beta} + \frac{A}{2\alpha} + \left( \frac{B}{2\beta} - \frac{A}{2\alpha} \right) (\beta - \alpha)x.
\end{align*}
$g_A(x)$ is an increasing function of $x$ as long as $0<A<\alpha$. It is also an increasing function of $A$  when $x<1$ and $\alpha < 1/2$.

Now, let us evaluate $E[N_{\theta^*}(j+1)|\ba(j)]$. For the three cases we have 
\begin{enumerate}[(i)]
\item When $u(j) = u(\theta^*), $ \\
$E[N_{\theta^*}(j+1)|\ba(j)] = P_1(j) g_{p_{j+1}} \left( \frac{\tau_1(j) - a_{u(\theta^*)}(j)}{1-a_{u(\theta^*)}(j)} \right) + P_2(j) g_0 \left( \frac{\tau_2(j) - a_{u(\theta^*)}(j)}{1-a_{u(\theta^*)}(j)} \right). $
\item When $u(j) > u(\theta^*), $ \\
$E[N_{\theta^*}(j+1)|\ba(j)] = P_1(j) g_{0} \left( \frac{-\tau_1(j) - a_{u(\theta^*)}(j)}{1-a_{u(\theta^*)}(j)} \right) + P_2(j) g_0 \left( \frac{\tau_2(j) - a_{u(\theta^*)}(j)}{1-a_{u(\theta^*)}(j)} \right). $
\item When $u(j) < u(\theta^*), $ \\
$E[N_{\theta^*}(j+1)|\ba(j)] = P_1(j) g_{p_{j+1}} \left(  \frac{\tau_1(j) - a_{u(\theta^*)}(j)}{1-a_{u(\theta^*)}(j)} \right) + P_2(j) g_{p_{j+1}} \left( \frac{-\tau_2(j) - a_{u(\theta^*)}(j)}{1-a_{u(\theta^*)}(j)} \right)$.
\end{enumerate}

We will now bound $E[N_{\theta^*}(j+1)|\ba(j)]$ for all the three cases. We will use the fact that for all $0<a<1$, we have $\frac{\tau-a}{1-a} \leq \tau$ and $-\left( \frac{\tau+a}{1-a} \right) \leq -\tau $.

Starting with case (i), we have 
$\tau_2(j) - a_{u(\theta^*)}(j) = a_{u(\theta^*)}(j) - \tau_1(j)$. Thus, 

\begin{align*}
E[N_{\theta^*}(j+1)|\ba(j)] & = P_1(j) g_{p_{j+1}} \left(  \frac{\tau_1(j) - a_{u(\theta^*)}(j)}{1-a_{u(\theta^*)}(j)} \right) + P_2(j) g_{0} \left( \frac{a_{u(\theta^*)}(j) - \tau_1(j)}{1-a_{u(\theta^*)}(j)} \right), \\
& \leq P_1(j) g_{p_{j+1}} \left(  \frac{\tau_1(j) - a_{u(\theta^*)}(j)}{1-a_{u(\theta^*)}(j)} \right) + P_2(j) g_{p_{j+1}} \left( \frac{a_{u(\theta^*)}(j) - \tau_1(j)}{1-a_{u(\theta^*)}(j)} \right), \\
& = \frac{q_{j+1}}{2\beta} + \frac{p_{j+1}}{2\alpha} + \left( \frac{q_{j+1}}{2\beta} - \frac{p_{j+1}}{2\alpha} \right)(\beta - \alpha)\frac{\tau_1(j) - a_{u(\theta^*)}(j)}{1-a_{u(\theta^*)}(j)}\left( P_1(j) - P_2(j) \right),\\
& = \frac{q_{j+1}}{2\beta} + \frac{p_{j+1}}{2\alpha} + \left( \frac{q_{j+1}}{2\beta} - \frac{p_{j+1}}{2\alpha} \right)(\beta - \alpha)\frac{\tau_1(j) - a_{u(\theta^*)}(j)}{1-a_{u(\theta^*)}(j)}\frac{\tau_2(j) - \tau_1(j)}{\tau_2(j) + \tau_1(j)}, \\
& = \frac{q_{j+1}}{2\beta} + \frac{p_{j+1}}{2\alpha} + \left( \frac{q_{j+1}}{2\beta} - \frac{p_{j+1}}{2\alpha} \right)(\beta - \alpha)\frac{\tau_1(j) - a_{u(\theta^*)}(j)}{1-a_{u(\theta^*)}(j)}\frac{2a_{u(\theta^*)}(j) - 2\tau_1(j)}{\tau_2(j) + \tau_1(j)}, \\
& \leq \frac{q_{j+1}}{2\beta} + \frac{p_{j+1}}{2\alpha}.
\end{align*}

In case (ii),
\begin{align*}
E[N_{\theta^*}(j+1)|\ba(j)] & \leq   P_1(j) g_{0}(-\tau_1(j)) + P_2(j) g_0(\tau_2(j)),  \\
& = P_1(j) \left(\frac{1}{2\beta} - \frac{1}{2\beta}(\beta - \alpha)\tau_1(j)\right) + P_2(j)\left(\frac{1}{2\beta} + \frac{1}{2\beta}(\beta - \alpha)\tau_2(j)\right), \\
& = \frac{1}{2\beta} \ \left( \text{because $-P_1\tau_1(j) + P_2\tau_2(j) = 0 $} \right), \\
& \leq \frac{q_{j+1}}{2\beta} + \frac{p_{j+1}}{2\alpha} \ \left( \text{because } \alpha \leq \frac{1}{2} \right).
\end{align*}

In case (iii),
\begin{align*}
E[N_{\theta^*}(j+1)|\ba(j)] & \leq P_1(j) g_{p_{j+1}}(\tau_1(j)) + P_2(j) g_{p_{j+1}}(-\tau_2(j)), \\
& = \frac{q_{j+1}}{2\beta} + \frac{p_{j+1}}{2\alpha} + \left( \frac{q_{j+1}}{2\beta} + \frac{p_{j+1}}{2\alpha} \right) (\beta - \alpha)(P_1(j)\tau_1 - P_2(j)\tau_2), \\
& = \frac{q_{j+1}}{2\beta} + \frac{p_{j+1}}{2\alpha}.
\end{align*}

Thus, in all the cases $E[N_{\theta^*}(j+1)|\ba(j)] \leq \frac{q_{j+1}}{2\beta} + \frac{p_{j+1}}{2\alpha} \leq  \frac{Q_e}{2\beta} + \frac{P_e}{2\alpha}$.

Substituting this in the inequality for $E[M_{\theta^*}(T)]$, we get
$E[M_{\theta^*}(T)] =  M_{\theta^*}(0)\left\lbrace \frac{Q_e}{2\beta} + \frac{P_e}{2\alpha} \right\rbrace^T.$ Hence, 
\begin{align*}
P(|\theta_T - \theta^*| > \epsilon) & \leq M_{\theta^*}(0)\left\lbrace \frac{q}{2\beta} + \frac{p}{2\alpha} \right\rbrace^T ,\\
&  \leq \frac{1-p_{1}^{-1}\epsilon}{p_{1}^{-1}\epsilon} \left\lbrace \frac{q}{2\beta} + \frac{p}{2\alpha} \right\rbrace^T ,\\
& = \frac{p_1 - \epsilon}{\epsilon} \left\lbrace \frac{q}{2\beta} + \frac{p}{2\alpha} \right\rbrace^T. 
\end{align*}

By construction, $\hat{\theta}_{T} \geq \theta_T$. If $\hat{\theta}_{T} > \theta_T$, then $\theta^* > \hat{\theta}_{T} > \theta_T$. Thus, 
$$P(|\hat{\theta}_{T} - \theta^*| > \epsilon) \leq \frac{p_1 - \epsilon}{\epsilon} \left\lbrace \frac{q}{2\beta} + \frac{p}{2\alpha} \right\rbrace^T.$$
\qed
\endproof

\noindent\textbf{V: Time Complexity of \aheu \ }

\begin{lemma}
The number of iterations in \aheu \ is $\left\lceil \log \frac{p_1}{\epsilon - 2(\nu^2 + 2\nu)} \right \rceil$.
\end{lemma}
\proof{Proof.}
Let $I_j$ denote the size of the search interval in the $j$-th iteration of \aheu \ , i.e., $I_j = U_j - L_j$. As described there are three possible updates of the search interval. In the first two updates, $I_{j+1} = \frac{I_j}{2}$. In the third update rule, 
\begin{align*}
I_{j+1} &= U_j - \hat{K}, \\
&= \frac{I_j}{2} + (\nu^2 + 2\nu)(1-K_j), \\
& \leq \frac{I_j}{2} + (\nu^2 + 2\nu).
\end{align*}

Define $\hat{\nu} = \nu^2 + 2\nu$. Then, in all the three cases $I_{j+1} \leq \frac{I_j}{2} + \hat{\nu}$.
Similarly, 
\begin{align*}
I_{j} \  & \leq  \ \frac{I_{j-1}}{2} + \hat{\nu}, \\ 
I_{j-1} \ & \leq \ \frac{I_{j-2}}{2} + \hat{\nu},\\ 
\vdots \\
I_1 \ & \leq \ \frac{I_{0}}{2} + \hat{\nu}.
\end{align*}

where $I_0 = p_1$.

Taking a telescopic sum, we get
\begin{align*}
I_j \ & \leq \ \frac{I_{0}}{2^t} + \hat{\nu} \left( 1 + \frac{1}{2} + \cdots \frac{1}{2^{j-1}} \right), \\
& = \ \frac{I_{0}}{2^t} + 2\hat{\nu} \left( 1- \frac{1}{2^t} \right), \\
& \leq  \ \frac{I_{0}}{2^t} + 2\hat{\nu}.
\end{align*}

Thus, the number of iterations required to get the size of the search interval to within $\epsilon$ is $\left\lceil \log_2 \frac{p_1}{\hat{\nu}}\right\rceil$.

\endproof
\end{document}